\documentclass{amsart}
\usepackage{amssymb}
\usepackage{amsfonts}
\usepackage{graphicx}
\usepackage{amscd}
\usepackage{amsmath}
\usepackage{epsfig}
\usepackage{epstopdf}
\usepackage{graphicx}
\usepackage{booktabs}
\usepackage{subcaption}

\newtheorem{theorem}{Theorem}
\theoremstyle{plain}

\newtheorem{definition}{Definition}
\newtheorem{example}{Example}

\newtheorem{remark}{Remark}

\numberwithin{equation}{section}

\begin{document}
\title[Remarks on $b$-enriched nonexpansive mappings]{ Remarks on $b$%
-enriched nonexpansive mappings}
\author{Rizwan Anjum$^{*}$}
\address{Rizwan Anjum, Department of Mathematics, Division of Science and
Technology, University of Education, Lahore 54770, Pakistan.}
\email{rizwananjum1723@gmail.com}
\author{Mujahid Abbas}
\address{Department of Mechanical Engineering Sciences, Faculty of Engineering and the Built Environment, Doornfontein Campus, University of Johannesburg, South Africa}
\email{mujahida@uj.ac.za}
\keywords{fixed point, fixed point theorem, $b$-enriched nonexpansive\\
Corresponding author: Rizwan Anjum (rizwananjum1723@gmail.com) }

\begin{abstract}
In this note, we analyzed the concept of enriched nonexpansive which was proposed  in ``Approximating fixed points of enriched
nonexpansive mappings by Krasnoselskij iteration in Hilbert spaces''
(Carpathian J. Math., 35(2019), No. 3, 293-304.) Through our analysis, we
conclude that the idea of enriched nonexpansive needs reconsideration, as it
coincides with well known concept of nonexpansive. Our
findings provide an insights into the existing literature and
highlight the need for further investigations and clarifications in the existing literature on a metric-fixed
point theory.
\end{abstract}

\maketitle

\section{Introduction}

During the 1960s, there was a notable upsurge in the investigation of
nonexpansive mappings, largely propelled by two significant factors. One
catalyst was Browder's exploration of the interplay between monotone
operators and nonexpansive mappings \cite{Browder1,Browder2,Browder3,Browder4,Browder5,Browder6}, which sparked
considerable interest in the field. Additionally, Kirk's seminal paper \cite%
{kirk} investigated the crucial role of the geometric properties of the norm
in establishing the existence of fixed points for nonexpansive mappings.

Let $C$ be a nonempty closed and convex subset in a real Hilbert space $H$
equipped with a norm $\Vert \cdot \Vert $ derived from an inner product $%
\big<\cdot ,\cdot \big>$. A mapping $T:C\rightarrow C$ is called
nonexpansive if it satisfies the contractive condition given by:
\begin{equation}
\Vert Tx-Ty\Vert \leq \Vert x-y\Vert ,\ \ x,y\in C.  \label{nonex}
\end{equation}

The study of nonexpansive mappings is a significant and actively researched
topic in nonlinear analysis. The reader interested in fixed point theory of
nonexpansive mappings is referred to \cite%
{Berinde,Chidume,Geobel,R22,R222,R2222,s} and references mentioned therein.%
\newline

Next, we state the Browder-Goede-Kirk fixed point theorem for nonexpansive
mappings, originally presented as Theorem 4 in \cite{Browder6} and
reproduced as Theorem 3.1 in \cite{Berinde}.

\begin{theorem}
\label{BGK} Let $C$ be a closed bounded convex subset of the Hilbert space $%
H $ and $T:C\rightarrow C$ a nonexpansive mapping. Then $T$ has at least one
fixed point.
\end{theorem}

Berinde in \cite{9} introduced and examined a new class of mappings known as
the class of enriched nonexpansive mappings, which broadens the scope of
nonexpansive mappings.

\begin{definition}
\cite{9}\label{Def} Let $\big(X,\|\cdot\|\big)$ be a normed space. A mapping
$T:X\rightarrow X$ is called an enriched nonexpansive or $b$-enriched
nonexpansive if there exists $b\in[0,+\infty)$ such that the following condition holds:
\begin{equation}  \label{enrichednon}
\left\Vert b(x-y)+Tx-Ty\right\Vert\leq (b+1) \left\Vert x-y\right\Vert,\ \
\forall \ x,y\in X.
\end{equation}
\end{definition}

It was stated in \cite{9} that the class of enriched nonexpansive mappings
is larger than the class of nonexpansive mappings. Indeed, if we take $b=0$
in (\ref{enrichednon}), then we obtain (\ref{nonex}).\newline

In \cite{9}, it was demonstrated that every enriched nonexpansive mapping
defined on a Hilbert space possesses a fixed point that can be approximated
using the Krasnoselskij iterative scheme employing the condition of the
additional property of demicompactness of the enriched nonexpansive mapping $%
T$.\newline
Now, we state the main result in \cite{9}.

\begin{theorem}
\cite{9}\label{Thm1} Let $C$ be a closed bounded convex subset of the
Hilbert space $H$ and $T:C\rightarrow C$ a $b$-enriched nonexpansive
mapping. Then $T$ has at least one fixed point.
\end{theorem}

By conducting a search on MathSciNet, it is observed that there are 105
papers indexed with the term "enriched nonexpansive" in their titles (see
\cite{1,2,3,4,5,6,7,8,9,10,11,12,13,14,15,16,17,18,19,20,21,22,23,BerindeNEw}).
Additionally, there are 2,350 papers that mentioned the term "enriched
nonexpansive" anywhere in the text, as of March 16, 2024.

The main objective of this paper is to demonstrate that the class of
enriched nonexpansive mappings does not contribute much, as it does not go
away from the class of nonexpansive mappings. Therefore, as a conclusion, we
\ suggest the researchers in the field of metric fixed point theory to
reconsider Theorem \ref{Thm1} for additional exploration. Our results
investigate significant review of the current literature and the need to
reconsider the continued investigation in this direction.

\section*{Remarks on $b$-enriched nonexpansive mappings}

In this section, we prove that the fixed point of enriched nonexpansive
mappings can be obtained by using the fixed point theory of nonexpansive
mappings.\newline

The subsequent remark provides the justification of the above comment.

\begin{remark}
\label{wnu} Any $b$-enriched nonexpansive (\ref{enrichednon}) can be
simplified as follows:
\begin{align*}
\left\Vert b(x-y)+Tx-Ty\right\Vert & \leq (b+1)\left\Vert x-y\right\Vert \\
\frac{1}{b+1}\left\Vert b(x-y)+Tx-Ty\right\Vert & \leq \left\Vert
x-y\right\Vert \\
\left\Vert \frac{b(x-y)+Tx-Ty}{b+1}\right\Vert & \leq \left\Vert
x-y\right\Vert \\
\left\Vert \frac{bx+Tx}{b+1}-\frac{by+Ty}{b+1}\right\Vert & \leq \left\Vert
x-y\right\Vert .
\end{align*}%
Therefore, the $b$-enriched nonexpansive condition (\ref{enrichednon})
reduces to the condition:
\begin{equation}
\left\Vert Sx-Sy\right\Vert \leq \left\Vert x-y\right\Vert ,\text{ for all }%
x,y\in X  \label{NewSe}
\end{equation}%
where $S:X\rightarrow X$ is defined as:
\begin{equation}
Sx=\frac{bx+Tx}{b+1},\ \ \ \forall x\in X.  \label{Sxxxx}
\end{equation}%
It follows from (\ref{NewSe}) that the mapping $S$ defined in (\ref{Sxxxx})
is a nonexpansive mapping.
\end{remark}

In conclusion, the $b$-enriched nonexpansive condition (\ref{enrichednon})
can be rephrased to a nonexpansive mapping contractive condition. We refer
readers to the proof of Theorem \ref{Thm1}, where it was established that
every $b$-enriched nonexpansive mapping possesses at least one fixed point (
see, Theorem 2.2 of \cite{9}). Through Remark \ref{wnu}, we show that the
set of fixed points of $S$ is equal to the set of fixed point of $T$.

\begin{remark}
\label{rem2} From Remark \ref{wnu}, it can be observed that every $b$%
-enriched nonexpansive mapping $T$ can be transformed into a nonexpansive
mapping $S$ as defined in (\ref{Sxxxx}). By applying Theorem \ref{BGK}, we
know that $S$ possesses at least one fixed point, denote it by $x^{\ast }\in
C$. This implies:
\begin{align*}
Sx^{\ast }& =\frac{bx^{\ast }+Tx^{\ast }}{b+1}=x^{\ast } \\
bx^{\ast }+Tx^{\ast }& =(b+1)x^{\ast } \\
bx^{\ast }+Tx^{\ast }& =bx^{\ast }+x^{\ast } \\
Tx^{\ast }& =x^{\ast }.
\end{align*}%
Therefore, $x^{\ast }$ is indeed the fixed point of $T$, thereby proving
Theorem \ref{Thm1}.
\end{remark}

\section*{Modification of $b$-enriched nonexpansive mappings}

In this section, we will give the modified definition of $b$-enriched
nonexpansive mappings and demonstrate fixed point theorems for such mappings.%
\newline
We now present the following concept.

\begin{definition}
\label{DefiNew} Let $\big(X,\Vert \cdot \Vert \big)$ be a normed space. A
mapping $T:X\rightarrow X$ is called an modified enriched nonexpansive or $b$%
-modified enriched nonexpansive if  there exists $b\in[0,+\infty)$ such that the following condition holds:
\begin{equation}
\left\Vert b(x-y)+Tx-Ty\right\Vert \leq \left\Vert x-y\right\Vert ,\ \
\forall \ x,y\in X.  \label{modifenrichednon}
\end{equation}
\end{definition}

The remark below illustrates that the definition of $b$-modified enriched
nonexpansive, as in \ref{DefiNew}, aligns accurately with the definition of
enriched nonexpansive as specified in the Definition \ref{Def}.

\begin{remark}
Any $b$-modified enriched nonexpansive (\ref{modifenrichednon}) can be
simplified as follows:
\begin{equation}
\left\Vert Sx-Sy\right\Vert \leq \left\Vert x-y\right\Vert ,\text{ for all }%
x,y\in X  \label{NewSess}
\end{equation}%
where $S:X\rightarrow X$ is defined as:
\begin{equation}
Sx=bx+Tx,\ \ \ \forall x\in X.  \label{Sxxxxss}
\end{equation}%
We are interested in the conditions under which a mapping $T$ satisfying (%
\ref{modifenrichednon}) has at the least one fixed point. It has been
observed that the contractive condition satisfied by $b$-modified enriched
nonexpansive can be reformulated as shown in (\ref{NewSess}) to obtain a
nonexpansive mappings $S$ (\ref{Sxxxxss}).  The Browder-Goede-Kirk
fixed-point theorem \ref{BGK} guarantees the existence of  at least one
fixed point of a nonexpansive mapping under certain conditions. Note that,
if $x^{\ast }\in X$ is the fixed point of $S$, then  $x^{\ast
}(1-b)=Tx^{\ast }$ implies that $x^{\ast }$ is not the fixed point of $T$.
This was not the case with the definition of enriched nonexpansive mappings
as specified in Definition \ref{Def} and mentioned in Remarks \ref{wnu} and %
\ref{rem2}.
\end{remark}

Clearly, the class of $b$-modified enriched nonexpansive mappings includes
the class of enriched nonexpansive mappings; that is, every mapping
satisfying (\ref{nonex}) reduces to (\ref{modifenrichednon}) for $b=0.$%
\newline
Now, we present the examples of mapping which are $b$-modified enriched
nonexpansive mappings but not a nonexpansive mappings.

\begin{example}
Let $X=\mathbb{R}$ be the set of real numbers equipped with the usual norm,
and  $T:X\rightarrow X$ be defined by $Tx=100-2x$ for all $x\in \mathbb{R}$.
Obviously, $T$ is a $3$-modified enriched nonexpansive mapping. Indeed, for
all $x,y\in \mathbb{R}$, we have:
\begin{align*}
\left\Vert 3(x-y)+Tx-Ty\right\Vert & =\left\Vert
3(x-y)+(100-2x)-(100-2y)\right\Vert  \\
& =\left\Vert 3x-3y+100-2x-100+2y\right\Vert  \\
& =\left\Vert x-y\right\Vert .
\end{align*}
\end{example}

\begin{example}
Let $X=\mathbb{R}$, and $b\in (0,+\infty )$ is a fixed real number. Let $%
T_{b}:X\rightarrow X$ be defined by $T_{b}x=(1-b)x$ for all $x\in \mathbb{R}$%
. Then $T_{b}$ is a $b$-modified enriched nonexpansive mapping. Indeed, for
all $x,y\in \mathbb{R}$, we have:
\begin{align*}
\left\Vert b(x-y)+T_{b}x-T_{b}y\right\Vert & =\left\Vert
b(x-y)+(1-b)x-(1-b)y\right\Vert  \\
& =\left\Vert bx-by+x-bx-y+by\right\Vert  \\
& =\left\Vert x-y\right\Vert .
\end{align*}
\end{example}

Our first result shows that in order to prove the existence of fixed points
of $b$-modified enriched nonexpansive mappings, we do not require the rich
geometric structure of a nonempty closed and convex subset in a real Hilbert
space, but this only holds when $b\neq 0.$ Furthermore, a unique fixed point
exists. We establish these results within the framework of a Banach space.%
\newline
Before we present our main theorem, we need the following concept from \cite%
{7}.

\begin{remark}
\label{7sssssssssssssss}\cite{7} If $T$ is a self mapping on a normed space
space $X$, then for any $\lambda \in (0,1)$, an averaged mapping $T_{\lambda
}$ given by
\begin{equation}
T_{\lambda }x=(1-\lambda )x+\lambda Tx,\ \forall x\in X.  \label{averg}
\end{equation}%
Moreover,
\begin{equation}
Fix(T_{\lambda })=\{x\in X:T_{\lambda }x=x\}=\{x\in X:Tx=x\}=Fix(T).
\label{sskkahahshh}
\end{equation}
\end{remark}

We start with following result.

\begin{theorem}
\label{MainThem1} Let $\left( X,\left\Vert \cdot \right\Vert \right) $ be
Banach space and $T:X\rightarrow X$ a $b$-modified enriched nonexpansive
mapping with $b\neq 0.$ Then

\begin{enumerate}
\item $Fix(T)=\{x^{*}\};$

\item there exists $\lambda \in (0,1) $ such that the Krasnoselskii
iteration associated to $T,$ that is, the sequence $\{x_{n}\}_{n=0}^{\infty
},$ given by
\begin{equation}
x_{n+1}=(1-\lambda) x_{n}+\lambda Tx_{n},\ \ n\geq 0,  \label{equationsss}
\end{equation}%
converges to $x^{\ast }$ for any initial guess $x_{0}\in X.$
\end{enumerate}
\end{theorem}

\begin{proof}
Let us denote $\lambda =\frac{1}{b+1}.$ Obviously, $\lambda \in (0,1)$ and
the $b$-modified enriched nonexpansive (\ref{modifenrichednon}) becomes
\begin{equation*}
\left\Vert (1-\lambda )x+Tx-((1-\lambda )y+Ty)\right\Vert \leq \lambda
\left\Vert x-y\right\Vert ,\ \ \forall \ x,y\in X,
\end{equation*}%
which can be written in an equivalent form as
\begin{equation}
\left\Vert T_{\lambda }x-T_{\lambda }y\right\Vert \leq \lambda \left\Vert
x-y\right\Vert ,\ \ \forall \ x,y\in X,  \label{ssssssssssssssssss}
\end{equation}%
where $T_{\lambda }$ is the averaged mapping defined in (\ref{averg}). Since
$\lambda \in (0,1)$,  (\ref{ssssssssssssssssss}) gives that $T_{\lambda }$
is a Banach contraction mapping.\newline
In view of (\ref{averg}), the Krasnoselskij iterative process $%
\{x_{n}\}_{n=0}^{\infty },$ defined by (\ref{equationsss}) is exactly the
Picard iteration associated with $T_{\lambda }$, that is,
\begin{equation}
x_{n+1}=T_{\lambda }x_{n},n\geq 0.  \label{Pics}
\end{equation}%
Since $X$ is a Banach space and $T_{\lambda }$ is a Banach contraction
mapping on $X$, it follows from the Banach fixed-point theorem (Theorem 2.1
in \cite{Berinde}) in the setting of Banach spaces that $T_{\lambda }$ has a
unique fixed point, denoted as $x^{\ast }\in X$, and the Picard iteration
associated with $T_{\lambda }$ as defined in (\ref{Pics}) converges to $%
x^{\ast }.$\newline
Also, noting from (\ref{sskkahahshh}), we observe that $T$ has a unique
fixed point since $T_{\lambda }$ does. This completes the proof.
\end{proof}

From Theorem \ref{MainThem1}, we do not derive Theorem \ref{BGK} as a direct
consequence. This is because Theorem \ref{MainThem1} is established under
the assumption that $b\neq 0$. The rationale behind imposing the condition $%
b\neq 0$ in Theorem \ref{MainThem1} is to establish a minimal requirement on
the structure of $X$ to ensure that every $b$-modified enriched nonexpansive
mapping $T$ defined on $X$ possesses at least one fixed point. This
condition is met when $X$ is a Banach space. To extend this conclusion to
all $b\in[0,+\infty)$, we present our subsequent theorem within the
framework of Hilbert spaces.

\begin{theorem}
\label{Main2}  Let $C$ be a closed bounded convex subset of the Hilbert
space $H $ and $T:C\rightarrow C$ a $b$-modified enriched nonexpansive
mapping. Then $T$ has at least one fixed point.
\end{theorem}

\begin{proof}
We divide the proof into the following two cases.\newline
Case 1. $b=0.$ In this case, the $b$-modified enriched nonexpansive mapping (%
\ref{modifenrichednon}) becomes
\begin{equation}
\left\Vert Tx-Ty\right\Vert \leq \left\Vert x-y\right\Vert ,\ \ \ \forall \
x,y\in C,  \label{b}
\end{equation}%
That is, $T$ is an nonexpansive mapping and hence by Theorem \ref{BGK}, $T$
has at least one fixed point.\newline
Case $2$. If $b>0$. Then $\lambda =\frac{1}{b+1}\in (0,1).$ Following
arguments similar to those given in the proof of Theorem \ref{MainThem1}, we
establish that $T_{\lambda }$ is a Banach contraction on $C$. Furthermore,
since $C$ is a closed subset of a Hilbert space, it is complete. The result
then follows by the Banach fixed-point theorem. This implies that $%
T_{\lambda }$ possesses a unique fixed point, as indicated by (\ref%
{sskkahahshh}). Therefore, $T$ also possesses a unique fixed point.
\end{proof}

\begin{remark}
We can derive Theorem \ref{BGK} by setting $b=0$ in Theorem \ref{Main2}.
This results in obtaining \ref{BGK} as a corollary of our main result.
\end{remark}

\section*{Competing Interests}

The authors declare that they have no competing interests.

\end{document}